\documentclass[10pt]{amsart}
\usepackage{graphicx}
\usepackage{amssymb}
\usepackage{amsmath,amsthm}
\usepackage{hyperref}
\usepackage{amsthm}
\usepackage[english]{babel}
\usepackage{mathrsfs}

\newtheorem{thm}{Theorem}[section]
\newtheorem{cor}[thm]{Corollary}
\newtheorem{lem}[thm]{Lemma}
\newtheorem{prop}[thm]{Proposition}
\theoremstyle{definition}
\newtheorem{defn}[thm]{Definition}
\newtheorem{rem}[thm]{Remark}

\newtheorem{fact}[thm]{Fact}

\newtheorem{prob}[thm]{Problem}

\newtheorem{question}[thm]{Question}

\def \R{\mathcal R}
\def \kop{\mathcal{K}_{op}}
\def \sp{\operatorname{span}}

\makeatletter

\def\dotminussym#1#2{%
  \setbox0=\hbox{$\m@th#1-$}%
  \kern.5\wd0%
  \hbox to 0pt{\hss\hbox{$\m@th#1-$}\hss}%
  \raise.6\ht0\hbox to 0pt{\hss$\m@th#1.$\hss}%
  \kern.5\wd0}
\newcommand{\dotminus}{\mathbin{\mathpalette\dotminussym{}}}

\title[]{Games and elementary equivalence of $\rm II_1$ factors}
\author{Isaac Goldbring and Thomas Sinclair}
\thanks{Goldbring's work was partially supported by NSF grant DMS-1007144. Sinclair was supported by an NSF RTG Assistant Adjunct Professorship.}

\address {Department of Mathematics, Statistics, and Computer Science, University of Illinois at Chicago, Science and Engineering Offices M/C 249, 851 S. Morgan St., Chicago, IL, 60607-7045}
\email{isaac@math.uic.edu}
\urladdr{http://www.math.uic.edu/~isaac}

\address{Department of Mathematics, University of California Los Angeles, 520 Portola Plaza Box 951555, Los Angeles, CA 90095-1555}
\email{thomas.sinclair@math.ucla.edu}
\urladdr{http://www.math.ucla.edu/~thomas.sinclair}

\date{\today}

\dedicatory{}
\newcommand{\e}{\varepsilon}

\newcommand{\bb}{\mathbb}

\newcommand{\cal}{\mathcal}
\newcommand{\fr}{\mathfrak}

\newcommand{\om}{\omega}
\newcommand{\uu}{\mathcal U}
\newcommand{\DL}{\operatorname{DL}}
\newcommand{\BP}{\operatorname{BP}}
\newcommand{\loc}{\operatorname{loc}}
\newcommand{\vN}{\operatorname{vN}}

\newcommand{\id}{\operatorname{id}}

\DeclareMathOperator*{\tr}{tr}

\DeclareMathOperator*{\Th}{Th}

\newcommand{\ip}[2]{\langle #1, #2 \rangle}
\providecommand{\abs}[1]{\lvert #1 \rvert}
\providecommand{\nor}[1]{\lVert #1 \rVert}

\begin{document}
\begin{abstract}
We use Ehrenfeucht-Fra\"iss\'e games to give a local geometric criterion for elementary equivalence of II$_1$ factors.  We obtain as a corollary that two II$_1$ factors are elementarily equivalent if and only their unitary groups are elementarily equivalent as $\mathbb Z_4$-metric spaces.
 \end{abstract}

\maketitle


\section*{Introduction}

While most mathematicians are concerned in determining when two objects in their field are isomorphic, logicians tend to be concerned with the coarser notion of \emph{elementary equivalence}.  Two (classical) structures $M$ and $N$ are said to be elementarily equivalent if and only if, for any first-order sentence $\sigma$ (in the language appropriate to the study of $M$ and $N$), we have $\sigma$ is true in $M$ if and only if $\sigma$ is true in $N$.  For structures appearing in analysis, a continuous logic is used in which sentences can now take a continuum of ``truth'' values; the appropriate notion of elementary equivalence is that the truth values of all sentences are the same in both structures.

The model-theoretic study of tracial von Neumann algebras began in earnest in \cite{mtoa1}, \cite{mtoa2}, and \cite{mtoa3}.  At the moment, there are only three distinct elementary equivalence classes of II$_1$ factors known.  (This should not be so surprising as it took a while for many isomorphism classes of II$_1$ factors to be discovered and elementary equivalence is a much coarser notion.)  Indeed, it was observed in \cite{mtoa3} that Property ($\Gamma$) and the property of being McDuff are both elementary properties (for separable II$_1$ factors).  Thus, if we let $M_{\DL}$ be a separable II$_1$ factor that has Property ($\Gamma$) but is not McDuff (see \cite{DL}), then $M_{\DL}$, the hyperfinite II$_1$ factor $\R$ and the free group factor $L(\mathbb F_2)$ are mutually non-elementarily equivalent.  Amongst those studying II$_1$ factors from a model-theoretic point of view, it is widely agreed that there should be more than three elementary equivalence classes of II$_1$ factors; in fact, there should probably be continuum many elementary equivalence classes.  At the moment, we cannot even answer the question:  is $\R\otimes L(\mathbb F_2)$ elementarily equivalent to $\R$?  In order to accomplish these goals, we need more tools for understanding elementary equivalence of II$_1$ factors.

Ehrenfeucht-Fra\"iss\'e games have long been a tool in model theory for establishing that structures are elementarily equivalent.  In \cite{heinrich-henson}, the authors exhibit an Ehrenfeucht-Fra\"iss\'e-type game used to establish elementary equivalence for Banach spaces.  In this note, we adapt the game from \cite{heinrich-henson} and combine it with an argument of Kirchberg appearing in \cite{kirch} in order to characterize elementary equivalence for II$_1$ factors belonging to the class $\kop$ (to be defined below).  We should note that, currently, we do not know of a II$_1$ factor that does not belong to the class $\kop$ and the existence of such a factor would already lead to two new theories of II$_1$ factors!

Recall Dye's Theorem \cite{dye}, which states that any two factors not of type $\rm I_{2^n}$ (e.g., any two $\rm II_1$-factors) are isomorphic if and only if their unitary groups are isomorphic (even as discrete groups).  Combining Dye's Theorem with the Keisler-Shelah Theorem (which states that two structures are elementarily equivalent if and only if they have isomorphic ultrapowers) and the fact that the functors of taking ultrapowers and taking unitary groups commute, we see that two II$_1$ factors are elementarily equivalent if and only if their unitary groups are elementarily equivalent as metric groups (with respect to the $\ell_2$ metric).  Using the aforementioned Ehrenfeucht-Fra\"iss\'e games and some further arguments, our main result is that we can improve upon the previous sentence, essentially removing the group structure:

\begin{thm}
Suppose that $M$ and $N$ are II$_1$ factors belonging to the class $\kop$.  Then $M$ and $N$ are elementarily equivalent if and only if $U(M)$ and $U(N)$ are elementarily equivalent as $\mathbb Z_4$-metric \emph{spaces}.
\end{thm}

Here, by a $\mathbb Z_4$-metric space, we mean a metric space $X$ equipped with an action of $\mathbb Z_4$ on $X$ by isometries.  Unitary groups of von Neumann algebras will always be considered as $\mathbb Z_4$-metric spaces by having the generator of $\mathbb Z_4$ act by multiplication by $i$.

In this paper, we assume that the reader is familiar with some basic model theory and von Neumann algebra theory.  Good references for continuous model theory are \cite{bbhu} and \cite{mtoa2}; the latter is geared towards the model theoretic study of operator algebras.

All normed spaces are assumed to be over the complex numbers, $\bb C$. For a normed space $X$ we denote the closed unit ball $(X)_1 := \{x\in X : \nor{x}\leq 1\}$.

For the convenience of the reader, we now recall the original notion of Ehrenfeucht-Fra\"iss\'e games in the context of continuous logic.  This has not appeared in the literature but has appeared in some online lecture notes of Bradd Hart \cite{bradd}. Fix an arbitrary language $\cal L$ and atomic formulae $\varphi_1(\vec x),\ldots,\varphi_k(\vec x)$ in the variables $\vec x=(x_1,\ldots,x_n)$ and $\e>0$.  The Ehrenfeucht-Fraisse game $\fr G(\varphi_1,\ldots,\varphi_k,\e)$ is played with $\cal L$-structures $M$ and $N$ as follows:  First Player I chooses $a_1\in M$ or $b_1\in N$ respecting the sort of $x_1$.  Player II choose $b_2\in N$ or $a_2\in M$ respectively.  The players alternate in this manner until they have produced sequences $a_1,\ldots,a_n\in M$ and $b_1,\ldots,b_n\in N$.  Player II then wins the game if and only if, for each $i=1,\ldots,k$, we have $|\varphi_i(\vec a)^M-\varphi_i(\vec b)^N|\leq \e$.  It is then a theorem that $M\equiv N$ if and only Player II has a winning strategy in each $\fr G(\varphi_1,\ldots,\varphi_k,\e)$.

\section{The class $\kop$}

Given a C$^*$ algebra $A$, recall that its opposite algebra $A^{op}$ is the algebra obtained from $A$ by multiplying elements in the opposite order, that is, for $a,b\in A$, we have $a\cdot_{op}b:=b\cdot a$.  It is immediate that $A^{op}$ is once again a C$^*$ algebra.  Furthermore, if $A$ is a von Neumann algebra, then $A^{op}$ is also a von Neumann algebra.  Note also that if $(A_i \ : \ i\in I)$ is a family of C$^*$ algebras (resp. tracial von Neumann algebras) and $\uu$ is an ultrafilter on $I$, then $(\prod_\uu A_i)^{op}\cong \prod_\uu A_i^{op}$ via the identity map, where the ultraproduct is understood to be the usual C$^*$ algebra ultraproduct (resp. tracial ultraproduct).

Many of the naturally occurring tracial von Neumann algebras are isomorphic to their opposites, e.g. $\R$ and $L(G)$ ($G$ any group).  There are examples of tracial von Neumann algebras that are not isomorphic to their opposites (see \cite{Connes}).  During a seminar talk given by the first author at Vanderbilt University, Jesse Peterson asked whether or not the class of all tracial von Neumann algebras isomorphic to their opposites is an axiomatizable class.  While we do not know the answer to this question (although we suspect the answer is negative), the answer is positive if one replaces the word ``isomorphism'' by ``elementary equivalence'' as we show in the following:

\begin{prop}\label{kop}
The class of all tracial von Neumann algebras that are elementarily equivalent to their opposites is an elementary class.
\end{prop}

\begin{defn}
We let $\kop$ denote the class of all tracial von Neumann algebras elementarily equivalent to their opposites.
\end{defn}

\begin{proof}[Proof of Proposition \ref{kop}] We present a proof suggested to us by Todor Tsankov as well as independently by the anonymous referee. There are a collection of axioms for the class $\kop$:  for every term $t$, recursively define the term $t^{op}$ by   defining $(t_1\cdot t_2)^{op}:=t_2^{op}\cdot t_1^{op}$.  Then one can recursively define, for any formula $\varphi$, the formula $\varphi^{op}$, the key clause being the atomic formulae, where one replaces every occurrence of a term $t$ by the term $t^{op}$.  Then the conditions $|\sigma-\sigma^{op}|=0$, as $\sigma$ ranges over all sentences, axiomatizes the class $\kop$.
\end{proof}

We remark in passing that alternately by \cite[Proposition 5.14]{bbhu}, it suffices to show that $\kop$ is closed under isomorphisms, ultraproducts, and ultraroots. We leave it as an exercise to the reader to verify these properties for $\kop$.

Since $\R$ and $L(\mathbb F_2)$ are isomorphic to their opposites, they belong to $\kop$.  Moreover, the example $M_{\DL}$ of a II$_1$ factor with Property ($\Gamma$) that is not McDuff given by Lance and Dixmier  in \cite{DL} is also isomorphic to its opposite.  Thus, we have:

\begin{cor}
If there is a II$_1$ factor that does not belong to $\kop$, then there are at least five theories of II$_1$ factors.
\end{cor}

\begin{proof}
If $N$ is a II$_1$ factor that does not belong to $\kop$, then the theories of $N$ and $N^{op}$ differ from each other and from the three known theories of II$_1$ factors. 
\end{proof}

\begin{question}
Are there more ``explicit'' axioms for the class $\kop$?  Can one use typical model-theoretic preservation theorems to show that $\kop$ is universally axiomatizable or $\forall\exists$-axiomatizable?
\end{question}

\begin{question}
Is there a single sentence $\sigma$ such that, adding the condition ``$\sigma=0$'' to the axioms for II$_1$ factors gives an axiomatization of $\kop$?
\end{question}

A negative answer to the last question implies that there must be infinitely many elementary equivalence classes of II$_1$ factors not belonging $\kop$.  Indeed, if there are only finitely many elementary equivalence classes of II$_1$ factors not belonging to $\kop$, then the class of II$_1$ factors not belonging to $\kop$ is readily verified to be elementary as well, whence a typical compactness argument is used to show that the last question has a positive answer.

\section{Model theory of Banach pairs} In order to frame the main results of the paper in the next section on the model theory of $\rm II_1$-factors, we introduce a class of linear (unbounded) metric structures (``Banach pairs'') for which $\rm II_1$-factors will be the primary set of examples. The important fact which we will see is that the theory of a $\rm II_1$-factor regarded as a Banach pair will determine its theory as $\rm II_1$-factor. For this reason we feel it is justified to introduce this treatment, despite several existing approaches in the literature for dealing with linear metric structures, e.g., \cite{ben, bbhu, henson-moore}, with at least one treatment \cite{mtoa2} being devoted to C$^*$-algebras and tracial von Neumann algebras.

\begin{defn} A \emph{Banach pair} $(X,\cal C)$ consist of a normed space $X$ and a distinguished subset $\cal C\subset (X)_1$ which is:
\begin{itemize}
\item complete;
\item roundly convex, i.e., $\lambda x  + \mu y\in \cal C$ for all $x,y\in \cal C$ and $\lambda,\mu\in \bb C$ with $\abs{\lambda} + \abs{\mu}\leq 1$; and
\item generating, i.e., $\bigcup_n n\cdot \cal C = X$.
\end{itemize}
\end{defn}

\noindent The main examples of Banach pairs we will be interested in are where $X = M$, a tracial von Neumann algebra equipped with the $2$-norm $\nor{x}_2 := \tr(x^*x)^{1/2}$, and $\cal C= (M)_1$, the (norm) closed unit ball.

A Banach pair $(X,\cal C)$ can be intepreted as an structure for the following language $\cal L_{\BP}$:

\begin{itemize}
\item There is one sort each for $\bb C$ and $X$.
\item There is a sequence of domains of quantification $\cal C_n$ for $X$.
\item There are function symbols $\imath_{m,n}: \cal C_m\to \cal C_n$ for $m\leq n$ to be interpreted as the usual inclusion maps.
\item $X$ is given the usual complex normed space axioms.
\item Axioms which show $0_X\in \cal C_1\subset (X)_1$.
\item Axioms to show each $\cal C_n$ is roundly convex.
\end{itemize}

For a Banach pair $(X,\cal C)$, for $x\in X$ we define $\nor{x}_{\cal C} := \inf\{t>0 : x\in t\cdot\cal C\}$ which can be checked to be a Banach norm on $X$. However, note that $\nor{\,\cdot\,}_{\cal C}$ is a definable predicate if and only if it is uniformly continuous with respect to the usual norm. (In the case that $X$ is a tracial von Neumann algebra this will be the case if and only if $X$ is finite-dimensional.)

As an $\cal L_{\BP}$-structure the ultrapower $(X,\cal C)^\uu$ can be identified with the Banach pair $(X^\uu, \cal C^\uu)$ where $X^\uu$ is the quotient space of $\{(x_i) : \lim_\uu \nor{x_i}_{\cal C}<\infty\}$ modulo the subspace $\{(z_i) : \lim_\uu\nor{z_i}_{\cal C}<\infty,\ \lim_\uu\nor{z_i}=0\}$ and $\cal C^\uu\subset X^\uu$ is defined in the obvious way.

We say that two Banach pairs $(X,\cal C)$ and $(Y,\cal D)$ are isomorphic (written $(X,\cal C)\cong (X,\cal D)$) if they are isomorphic as $\cal L_{\BP}$-structures, that is, if there is an isometry $T: X\to Y$ so that $T(\cal C) = \cal D$. By definition, the aforementioned Banach pairs are elementarily equivalent (written $(X,\cal C)\equiv (Y,\cal D))$ if $\Th(X,\cal C) = \Th(Y, \cal D)$. As a consequence of the Keisler--Shelah theorem in continuous logic, we have that $(X,\cal C)\equiv (Y,\cal D)$ if and only if there is an ultrafilter so that $(X,\cal C)^\uu \cong (Y,\cal D)^\uu$. See \S10 in \cite{henson-iovino} for a proof of this fact in the context of normed spaces or \S3 in \cite{heinrich-henson} for a more explicit construction for Banach spaces.

Our main observation in this section is that for Banach pairs $(X,\cal C)$ and $(Y,\cal D)$ elementary equivalence can be characterized in terms of the pairs ``having the same local geometric structure'' by the use of Ehrenfeucht--Fra\"iss\'e games. For the very similar case of Banach spaces, this was done by Heinrich and Henson \cite[Theorem 4]{heinrich-henson} and the case of normed spaces is largely similar (see Remark 10.10 in \cite{henson-iovino}).

We now describe precisely what we mean when we say that two Banach pairs $(X,\cal C)$ and $(Y,\cal D)$ have the same local geometric structure. For $E$ a subspace of $X$ and $F$ a subspace of $Y$ we say that a linear bijection $T: E\to F$ is an \emph{$\e$-almost isometry} if $\nor{T},\nor{T^{-1}}\leq 1+\e$ and $T(E\cap \cal C) \subset_\e F\cap \cal D$ and $T^{-1}(F\cap \cal D)\subset_\e E\cap \cal C$. (We write $A\subset_\e B$ if $\sup_{x\in A}\inf_{y\in B}\nor{x - y}\leq \e$.)

The following is adapted from \S2 of \cite{heinrich-henson}: see also \S8 of \cite{henson-moore}. We describe a game $\fr G(n,\e)$ played by two players with Banach pairs $(X,\cal C)$ and $(Y,\cal D)$, where $\e>0$ and $n$ are fixed parameters.

\noindent {\bf Step 1.} Player I chooses a one-dimensional subspace, either $E_1\subset X$ or $F_1\subset Y$. Player II then chooses a subspace, respectively $F_1\subset Y$ or $E_1\subset X$ and a linear bijection $T_1: E_1\to F_1$.

\noindent{\bf Step i.} Player I chooses an at most one-dimensional extension, either $E_i \supset E_{i-1}$ or $F_i\supset F_{i-1}$. Player II then chooses a subspace, respectively $F_i\subset Y$ or $E_i\subset X$ and a linear bijection $T_i: E_i\to F_i$ which extends $T_{i-1}$.

\noindent{\bf Step n.} The players make their choices, and the game terminates. Player II wins if $T_n: E_n\to F_n$ is an $\e$-almost isometry; otherwise, Player I wins.

During the course of proofs, we may speak of Player I playing $x_i\in X$, in which case we mean that Player I plays $\sp(E_{i-1}\cup \{x_i\})$.  We may then also say that Player II responds with $y_i\in Y$, in which case we mean that Player II plays the linear bijection $T_i$ extending $T_{i-1}$ that sends $x_i$ to $y_i$.

\begin{defn} We say that Banach pairs $(X,\cal C)$ and $(Y,\cal D)$ are \emph{locally equivalent} (written $(X,\cal C)\cong_{\loc} (Y,\cal D)$) if for every $\e>0$ and every $n$, Player II has a winning strategy for the game $\fr G(n,\e)$.
\end{defn}

\begin{rem}\label{alternate} Since $\e$ is arbitrary, and we need only deal with at most one-dimensional extensions, we see that local isomorphism remains the same under an alternate version of $\e$-almost isometry, namely, the existence of linear bijections $T: E\to F$, $S:F\to E$ with strict containment $T(E\cap \cal C)\subseteq F\cap \cal D$ and $S(F\cap \cal D)\subseteq E\cap \cal C$ so that $\nor{ST - \id_E}$, $\nor{TS - \id_F}<\e$ and $\|T\|,\|S\|<1+\e$.
\end{rem}

\begin{prop}\label{EFBP} The following statements are equivalent:
\begin{enumerate}
\item $(X,\cal C)\equiv (Y,\cal D)$;
\item there exists an ultrafilter so that $(X,\cal C)^\uu\cong (Y,\cal D)^\uu$ as Banach pairs;
\item $(X,\cal C)\cong_{\loc} (Y,\cal D)$.
\end{enumerate}
\end{prop}

\noindent As noted above (1) $\Leftrightarrow$ (2) is the Keisler--Shelah theorem applied to the language of Banach pairs.  The proof of (2) $\Rightarrow$ (3) is straightforward using representing sequences. Therefore we only need to prove (3) $\Rightarrow$ (1).  The proof is more or less identical to the Banach space version as in \cite{heinrich-henson}.  However, since we are working in a different logic, we sketch a (nearly complete) proof here for the convenience of the reader.

\begin{proof}[Sketch of (3) $\Rightarrow$ (1)]
First, we work with the notion of $\epsilon$-almost isometry as described in Remark \ref{alternate}.  Let $\sigma$ be a sentence of the form $\inf_{v_1}\sup_{v_2}\cdots Q_{v_n} \rho(v_1,\ldots,v_n)$, where $Q$ is $\inf$ if $n$ is odd and $\sup$ if $n$ is even and where $\rho$ is quantifier-free.  (We suppress mention of the sorts $\cal C_i$ corresponding to each $v_i$.)  Fix $\epsilon>0$.  It suffices to show that $\sigma^{(Y,\cal D)}\leq \sigma^{(X,\cal C)}+\epsilon$ for all $\epsilon>0$.  Indeed, by symmetry of the relation of local equivalence, this shows that all sentences of the above form have the same truth values in $(X,\cal C)$ and $(Y,\cal D)$.  Since any sentence in prenex normal form is equivalent to one of the above form (by adding dummy variables) and since the set of  sentences in prenex normal form is dense in the set of all sentences (see \cite[\S6]{bbhu}), we obtain that $(X,\cal C)\equiv (Y,\cal D)$.

Fix sufficiently small $\delta>0$.  (We will see exactly how small $\delta$ needs to be in a moment.)  Fix a winning strategy $\cal S$ for Player II in $\fr G(n,\delta)$.  Call a play of the game $\fr G(n,\delta)$ \emph{regular} if:
\begin{itemize}
\item For odd $i$, Player I plays $x_i\in X$, while for even $i$, Player I plays $y_i\in Y$;
\item For each $i$, Player I's move at Round $i$ is always in the sort corresponding to the variable $v_i$;
\item Player II always plays according to $\cal S$.
\end{itemize}
We say that sequences $x_1,\ldots,x_k\in X$ and $y_1,\ldots,y_k\in Y$ \emph{correspond} if they are the results of the first $k$ rounds of a regular play of $\fr G(n,\delta)$.

For $0\leq l \leq n$, let $\sigma_l(v_1,\ldots,v_{n-l})$ denote the formula obtained from $\sigma$ by removing the first $n-l$ quantifiers.  One now proves, by induction on $l$ ($0\leq l\leq n$), that if $x_1,\ldots,x_{n-l}\in X$ and $y_1,\ldots,y_{n-l}\in Y$ correspond, then
$$\sigma_l(y_1,\ldots,y_{n-1})^{(Y,\cal D)}\leq \sigma_l(x_1,\ldots,x_{n-1})^{(X,\cal C)}+\e.$$  The base case $l=0$ follows from the fact that $T_n:\sp(x_1,\ldots,x_n)\to \sp(y_1,\ldots,y_n)$ is a $\delta$-almost isometry if $\delta$ is chosen sufficiently small.  We now prove the induction step.  Suppose that the claim holds for $l$ and that $x_1,\ldots,x_{n-l-1}\in X$ and $y_1,\ldots,y_{n-l-1}\in Y$ correspond.  Let $r:=\sigma_{l+1}(x_1,\ldots,x_{n-l-1})^{(X,\cal C)}$.  First suppose that $n-l$ is odd, so that $\sigma_{l+1}(v_1,\ldots,v_{n-l-1})=\inf_{v_{n-l}}\sigma_l(v_1,\ldots,v_{n-l})$.  Fix $\eta>0$ and let $x_{n_l}\in X$ be of the same sort as $v_{n-l}$ so that $\sigma_l(x_1,\ldots,x_{n-l})^{(X,\cal C)}\leq r+\eta$.  Let $y_{n-l}\in Y$ be Player II's response to $x_{n-l}$ according to the strategy $\cal S$.  Then, by induction,
$$\sigma_l(y_1,\ldots,y_{n-l})^{(Y,\cal D)}\leq \sigma_l(x_1,\ldots,x_{n-l})^{(X,\cal C)}+\e\leq r+\e+\eta.$$  Letting $\eta$ go to $0$ yields the desired result.  The case that $n-l$ is even is similar and is left to the reader.
\end{proof}


\section{Elementary equivalence of $\rm II_1$-factors}

We say that two tracial von Neumann algebras $M$ and $N$ are locally equivalent if the associated Banach pairs $(M,(M)_1)$ and $(N, (N)_1)$ are locally equivalent.  Somewhat miraculously, it turns out that for II$_1$ factors belonging to $\kop$, local equivalence is the same as elementary equivalence.  This essentially follows from an argument of Kirchberg in \cite{kirch}.  First, we need to recall a fact about \emph{Jordan morphisms} between von Neumann algebras.

Given a C$^*$ algebra $A$, the \emph{special Jordan product} on $A$ is the operation $\circ$ defined by $a\circ b:=\frac{1}{2}(ab+ba)$ for all $a,b\in A$.  If $B$ is also a C$^*$ algebra, then a linear map $T:A\to B$ is a \emph{Jordan morphism} if it preserves the special Jordan product and the involution.  We need the following:

\begin{fact}(See \cite[Corollary 7.4.9]{jordan})
If $M$ and $N$ are von Neumann algebras and $T:M\to N$ is a normal Jordan homomorphism, then $T$ is the sum of a $*$-homomorphism and a $*$-antihomomorphism.
\end{fact}

Recall that a map $A\to B$ between C$^*$ algebras is a $*$-antihomomorphism if and only if it is a $*$-homorphism $A\to B^{op}$.

Suppose that $M$ and $N$ are von Neumann algebras and $T:M\to N$ is a unital, bijective, normal Jordan homomorphism.  Write $T=T_1+T_2$, where $T_1:M\to N$ and $T_2:M\to N^{op}$ are $*$-homomorphisms.  Since $T_i(1)$ is a projection for $i=1,2$ and $T_1(1)+T_2(1)=1$, $T_1(1)$ and $T_2(1)$ are orthogonal projections.  Since $T(M)=N$, it follows that each $T_i(1)$ is a central projection.  Thus, if $N$ is a factor, it follows that $\{T_1(1),T_2(1)\}=\{0,1\}$, whence $T$ is either an isomorphism or an anti-isomorphism.

The following is basically Proposition 4.6 in \cite{kirch}.

\begin{prop}[Kirchberg]\label{jordan} Suppose that $M$ and $N$ are II$_1$ factors. If there is an isometry $T: L^2(M,\tr_M)\to L^2(N,\tr_N)$ so that $T$ maps $M$ onto $N$ contractively, then $M\cong N$ or $M\cong N^{op}$.
\end{prop}


\begin{proof} We first show that $T$ maps unitaries to unitaries. If $u\in M$ is a unitary, we have $1 = \nor{u}_2^2 = \nor{T(u)}_2^2 = \ip{T(u)}{T(u)} = \ip{T(u)^*T(u)}{1}$. On the other hand $\nor{T(u)^*T(u)}_2\leq \nor{T(u)}\cdot \nor{T(u)}_2\leq 1$.  It follows that $T(u)^*T(u) = 1$. We thus have that $T'(x) := T(1)^*T(x)$ is unital, contractive, trace-preserving, and takes unitaries to unitaries. By the same reasoning as in the proof of \cite[Proposition 4.6]{kirch}, $T'$ is a weakly-continuous Jordan morphism and the result follows from the discussion preceding this proposition.
\end{proof}


\begin{cor}
Suppose that $M$ and $N$ are II$_1$ factors.  Then $M$ is locally equivalent to $N$ if and only if $M$ is elementarily equivalent to $N$ or to $N^{op}$.  In particular, if $M$ and $N$ are II$_1$ factors belonging to the class $\kop$, then $M$ is locally equivalent to $N$ if and only if $M$ is elementarily equivalent to $N$.
\end{cor}

\begin{proof}
By the Downward L\"owenheim-Skolem Theorem (see \cite[Section 4.2]{mtoa2}), we may suppose that $M$ and $N$ are separable.  Suppose that $M$ is locally equivalent to $N$.  Then by Proposition \ref{EFBP}, there is an isometry $L^2(M^\uu)\to L^2(N^\uu)$ that maps $M^\uu$ into $N^\uu$ contractively.  By Proposition \ref{jordan}, $M^\uu$ is isomorphic to either $N^\uu$ or $(N^\uu)^{op}$.  It follows that $M$ is elementarily equivalent to either $N$ or $N^{op}$.  The converse is trivial.
\end{proof}

We now introduce a more useful test for determining elementary equivalence which works in the more specific case of Banach pairs $(M, (M)_1)$ where $M$ is a $\rm II_1$-factor (or more generally a tracial von Neumann algebra) equipped with the $2$-norm and $(M)_1$ is the (operator norm) unit ball of $M$.

We define the game $\fr G_{\vN}(n,\e)$ in parameters $n$ and $\e>0$ which is played by two players with $\rm II_1$-factors $M$ and $N$ as follows.

\noindent{\bf Step i.} Player I chooses a unitary either $u_i\in U(M)$ or $v_i\in U(N)$. Player II then chooses a unitary, respectively $v_i\in U(N)$ or $u_i\in U(M)$ in the same manner.

\noindent{\bf Step n.} The players make their choices, and the game terminates. Player II wins if $\abs{\ip{u_i}{u_j} - \ip{v_i}{v_j}}<\e$ for all $1\leq i,j\leq n$; otherwise, Player I wins.

\begin{thm}\label{local-ii} The $\rm II_1$-factors $M$ and $N$ are locally equivalent if and only if Player II has a winning strategy for the game $\fr G_{\vN}(n,\e)$ for all parameters $(n,\e)$.
\end{thm}

In order to prove this result we will first need one lemma.

\begin{lem}\label{unitary} Let $M$ and $N$ be $\rm II_1$-factors, $E\subset M$ and $F\subset N$ be subspaces, and $T: (E, E\cap (M)_1)\to (F,F\cap (N)_1)$ is an $\e$-almost isometry. If $u\in E$ is a unitary, then there exists a unitary $v\in N$ so that $\nor{T(u) - v}_2\leq 4\sqrt{\e}$.
\end{lem}

\begin{proof} In a $\rm II_1$-factor a $u$ is a unitary if and only if it is a contraction with $\nor{u}_2=1$. By definition, we see that there exists a contraction $y\in N$ with $\nor{y -T(u)}_2\leq \e$. In particular $\nor{y}_2\geq 1-2\e$. By a standard estimate we have that $$\nor{1 - \abs{y}}_2^2\leq  1 + \nor{\abs{y}}_2^2 - 2\tr(\abs{y}) = 1 + \tr(\abs{y}^2) -2\tr(\abs{y})\leq 1 -\tr(\abs{y})\leq 1 - \nor{y}_2\leq 2\e,$$ whence writing $y = v\abs{y}$ for $v\in U(N)$ we have that $\nor{T(u) -v}_2\leq 4\sqrt{\e}$.
\end{proof}

\begin{proof} [Proof of Theorem \ref{local-ii}]

First suppose that $M$ and $N$ are locally equivalent.  Fix $n$ and $\e>0$; we describe a winning strategy for Player II in the game $\frak G_{\vN}(n,\e)$.  For simplicity, we suppose that $n=2$ and describe a winning strategy for Player II; the general case is no more difficult, only the notation is more cumbersome.  Fix $\delta$ sufficiently small (to be specified later) and fix a winning strategy $\cal S$ for Player II in the game $\frak G(2,\delta)$.    Suppose that Player I first plays $u_1\in U(M)$.  (The case that Player I first plays a unitary in $N$ is similar.)  Let $y_1\in N$ be Player II's response to $u_1$ in the game $\frak G(2,\delta)$ according to $\cal S$.  Since $u_1\mapsto y_1$ determines a $\delta$-almost isometry, by Lemma \ref{unitary}, there is $v_1\in U(N)$ such that $\|y_1-v_1\|_2\leq 4\sqrt{\delta}$.  Now suppose that Player II responds with $v_2\in U(N)$.  (The case that Player II responds with a unitary in $M$ is similar.)  Let $x_2\in M$ be Player II's response to $(u_1,y_1,v_2)$ in the game $\frak G(2,\delta)$ according to $\cal S$.  Since $u_1\mapsto y_1$, $x_2\mapsto v_2$ determines a $\delta$-almost isometry, we once again have $u_2\in U(M)$ such that $\|x_2-u_2\|_2\leq 4\sqrt{\delta}$.

We need to verify that $|\langle u_i,u_j\rangle -\langle v_i,v_j\rangle |<\e$ for $i,j=1,2$.  If $\delta$ is chosen small enough so that a $\delta$-almost isometry preserves inner products within an error of $\frac{\e}{3}$ (use, for example, the Polarization Identity) and such that perturbing entries of an inner product by a distance of no more than $4\sqrt{\delta}$ changes the inner product by an amount not exceeding $\frac{\e}{3}$, then the desired estimates hold.  For example:
$$\langle u_1,u_2\rangle \sim_{\frac{\e}{3}} \langle u_1,x_2\rangle \sim_{\frac{\e}{3}} \langle y_1,v_2\rangle \sim_{\frac{\e}{3}} \langle v_1,v_2\rangle.$$


We now prove the converse.  Suppose that Player II has a winning strategy in all of the games $\frak G_{\vN}(n,\e)$; we show that $M$ and $N$ are elementarily equivalent as Banach pairs.  By symmetry, it is enough to show that $\sigma^{(M, (M)_1)}\leq r$ implies that $\sigma^{(N,(N)_1)}\leq r$ for any positive real number $r$ and any prenex normal form sentence $\sigma$.  Since $\sigma\dotminus r$ is equivalent to a prenex normal form sentence, it is enough to prove that $\sigma^{(M,(M)_1)}=0$ implies $\sigma^{(N,(N)_1)}=0$ for any prenex normal form sentence $\sigma$.

Towards this end, we introduce the ``unitary transform'' of a sentence in prenex normal form.  Suppose that $\sigma$ is a sentence in prenex normal form, say $$\sigma=Q_1x_1\cdots Q_n x_n \varphi(\vec x),$$ where $\varphi(\vec x)$ is quantifier-free.  We form the new sentence $\sigma^u$ as follows:
\begin{itemize}
\item If $Q_i=\inf$ and $x_i$ is of sort $n_i$, replace each occurrence of the variable $x_i$ by the term $t_i(u_i,v_i):=n_i\cdot (\frac{u_i+v_i}{2})$, where $u_i$ and $v_i$ are variables of sort $\cal C_1$, and replace the quantifier $Q_i x_i$ by the quantifiers $Q_i u_i Q_i v_i$.
\item The quantifier-free part of $\sigma^u$ should now be $$\max(\varphi,\max_i(\max(1\dotminus \|u_i\|_2,1\dotminus \|v_i\|_2))).$$
\end{itemize}

For example, if $\sigma=\sup_{x_1}\inf_{x_2} \varphi(x_1,x_2)$ where $x_2$ is of sort $\mathcal C_1$ (for simplicity), then $\sigma^u=\sup_{x_1}\inf_{u_2}\inf_{v_2} \varphi(x_1,\frac{u_2+v_2}{2})$.

Also, we let $\sigma^{uu}$ be the ``formula'' defined in the exact same way as $\sigma^u$ except that we only allow quantifiers over the unitary groups rather than the entire unit ball.  (Formally, $\sigma^{uu}$ is not a formula in the sense of continuous logic, but it will be useful in the remainder of the proof.)

\

\noindent \textbf{Claim 1:}  We have $\sigma^{(M,(M)_1)}=0$ if and only if $(\sigma^u)^{(M,(M)_1)}=0$ (and the corresponding statement for $(N,(N)_1)$).

\

\noindent Claim 1 follows from the fact that, in a \emph{finite} von Neumann algebra, any contraction is an average of two unitaries. Indeed if $x$ is a contraction in a finite von Neumann algebra, then it has polar decomposition $x = u |x|$, where $u$ is a unitary. As $|x|$ is a self-adjoint contraction, by functional calculus it may be written as the average of two unitaries.

\

\noindent \textbf{Claim 2:}  $(\sigma^u)^{(M,(M)_1)}=0$ if and only if $(\sigma^{uu})^{(M,(M)_1)}=0$ (and the corresponding statement for $(N,(N)_1)$).

\

\noindent The backwards direction of Claim 2 is trivial; the forwards direction follows from the fact that if $x$ is a contraction in a finite factor and $\nor{x}_2\geq 1-\e$, then there is a unitary $u$ so that $\nor{u - x}_2\leq 2\sqrt \e$.

\

\noindent Finally, suppose that $\sigma$ is a sentence in prenex normal form and $\sigma^{(M,(M)_1)}=0$.  Then by Claims 1 and 2, we have $(\sigma^{uu})^{(M,(M)_1)}=0$.  Since atomic formulae are of the form $\nor{\lambda_1x_1 + \dotsb + \lambda_nx_n}_2$ and arbitrary quantifier-free formulae are continuous combinations of atomic formulae, it follows from a winning strategy for Player II in $\frak G_{\vN}(n,\e)$ (for suitably small $\e$) that $(\sigma^{uu})^{(N,(N)_1)}=0$, whence $\sigma^{(N,(N)_1)}=0$ by Claims 1 and 2 again.

\end{proof}

Suppose now that $\cal L_i=\{\phi\}$, where $\phi$ is a unary function symbol with modulus of uniform continuity $\Delta_\phi(\e)=\e$.  If $M$ is a tracial von Neumann algebra, we view $U(M)$ as an $\cal L_i$-structure by interpreting $\phi$ as multiplication by $i$.  We then have:

\begin{cor} Let $M$ and $N$ be $\rm II_1$-factors.  Then $M$ and $N$ are locally equivalent if and only if $U(M)$ and $U(N)$ are elementarily equivalent as $\cal L_i$-structures.
\end{cor}

\begin{proof}
If $M$ and $N$ are locally equivalent, then $M$ is elementarily equivalent to either $N$ or $N^{op}$.  It follows that there is an ultrafilter $\uu$ such that $M^\uu$ is isomorphic to $N^\uu$ or $(N^{op})^\uu$.  In either case, $(U(M))^\uu=U(M^\uu)$ is isomorphic to $(U(N))^\uu=U(N^\uu)$ as $\cal L_i$-structures, whence $U(M)$ and $U(N)$ are elementarily equivalent as $\cal L_i$-structures.

Conversely, assume that $U(M)$ and $U(N)$ are elementarily equivalent as $\cal L_i$-structures.  Then Player II has a winning strategy for the EF-games for $U(M)$ and $U(N)$ as $\cal L_i$-structures.  It then follows that Player II has a winning strategy in the games $\frak G_{\vN}$ for $M$ and $N$.  Indeed, this follows from the fact that
$$\Re \langle u_i,u_j\rangle =1-\frac{1}{2}d(u_i,u_j)^2, \quad \Im \langle u_i,u_j\rangle=1-\frac{1}{2}d(u_i,i\cdot u_j)^2.$$
\end{proof}

\begin{rem}
Notice that the proof of the previous corollary gives an alternative proof of the forward direction of Theorem \ref{local-ii}.
\end{rem}

\begin{cor}\label{kop-ee}
Let $M$ and $N$ be $\rm II_1$ factors in the class $\kop$.  Then $M$ and $N$ are elementarily equivalent if and only if $U(M)$ and $U(N)$ are elementarily equivalent as $\cal L_i$-structures.
\end{cor}

\begin{cor} Let $M$ and $N$ be $\rm II_1$-factors. Suppose for every $\e$ that there is an $(1+\e)$-Lipschitz homeomorphism $f: U(M)\to U(N)$, that is, $f$ is bijective with $(1+\e)^{-1}\nor{u - v}_2\leq \nor{f(u) - f(v)}_2\leq (1+\e)\nor{u - v}_2$, that is further assumed to preserve the action by $\mathbb Z_4$. Then $M$ and $N$ are locally isomorphic.
\end{cor}

We will say that $M$ and $N$ are \emph{approximately Lipschitz isometric} if the condition of the previous corollary is satisfied. Although this relation ought to be in principle much stronger than elementary equivalence, to the best of our knowledge the results of \cite{mtoa3} heretofore furnish the only know examples of properties invariant under this relation namely, the McDuff property and property ($\Gamma$). It is, however, tempting to speculate that approximate Lipschitz isometry ought to be equivalent to isomorphism (up to opposites). 

In lieu of this, it would be highly interesting to determine whether hyperfiniteness is an invariant of approximate Lipschitz isometry. If true, this would be in contrast with \cite[Theorem 4.3]{mtoa3} which shows in particular that hyperfiniteness is not an invariant of elementary equivalence. Though one can show, essentially by Fact 3.1 and Proposition 3.2 (see also \cite[chapter XIV.2]{tak}), that for every $n$ there exists $\e>0$ so that any $\e$-approximate Lipschitz embedding $\theta$ of $M_n$ into a $\rm II_1$-factor $N$ there is a $\ast$-homomorphism $\theta': M_n\to N$ so that the image of the unit ball under $\theta$ is $\e$-contained in $2$-norm in the image unit ball under $\theta'$ of $M_n$, this still does not seem sufficient, unless $\e$ could be taken independent of $n$.

\section{Further remarks and open problems}

Of course, Corollary \ref{kop-ee} raises the question:  which $\mathbb Z_4$-metric spaces arise as unitary groups of II$_1$ factors? Even more importantly, what are the theories of such $\mathbb Z_4$-metric spaces?  Ignoring the extra structure for a moment, an important example of a complete theory of (noncompact) metric spaces is the theory of the \emph{Urysohn metric space}.  (See, for example, \cite{ealygold}.)  Recall that the Urysohn metric space is the unique (up to isometry) complete, separable metric space that is universal (that is, every separable metric space isometrically embeds) and ultrahomogeneous (every isometry between finite--even compact--subspaces extends to an isometry of the entire space).  However, the Urysohn space (or rather, its bounded counterpart, the Urysohn sphere) could never be isometric to the unitary group of a II$_1$ factor as the latter's metric is always negative definite.

Note that for $M$ with separable predual, $U(M)$ isometrically embeds naturally in $\mathbb S^\infty$, the Hilbert sphere in $\ell^2$.  The space $\mathbb S^\infty$ is the ``Hilbertian Urysohn sphere" in the sense described in \cite{nvt}, section 1.4.2.




It is well worth pointing out the following proposition, which is an immediate consequence of Ozawa's \cite{ozawa} fundamental result on the non-existence of a universal, separable $\rm II_1$ factor.

\begin{prop}\label{Z4universal} For any separable II$_1$ factor $M$, $U(M)$ is not universal among all $\bb Z_4$-metric spaces which embed (as $\mathbb Z_4$-metric spaces) in $\mathbb S^\infty$.
\end{prop}

\begin{proof} Suppose, towards a contradiction, that there is a II$_1$ factor $M$ for which $U(M)$ is universal among all $\bb Z_4$-metric spaces which embed in $\bb S^\infty$.  In particular, for any $\rm II_1$-factor $N$ with separable predual, $U(N)$ isometrically embeds in $U(M)$ in a way which commutes with the action of $i$. Since this embedding respects the inner product, it is not hard to see it must extend to an isometric embedding $L^2(N)\to L^2(M)$ which takes $N$ into $M$ contractively. Thus, as above, there is a unital injective $\ast$-homomorphism $N\hookrightarrow pMp\oplus ((1-p) M (1-p))^{op}$, whence $N$ embeds in either $M$ or $M^{op}$ since $N$ is a factor. However, this would contradict the fact \cite{ozawa} that there is no separable universal $\rm II_1$-factor (pick $M\star M^{op}$).
\end{proof}

\begin{question} Can $U(M)$ ever be universal among all metric spaces which embed in $\cal S^\infty$?
\end{question}

Proposition \ref{Z4universal} is good evidence that the answer to the previous question is no. We remark that a positive answer to the previous question would be equivalent to demonstrating the existence of a separable $\rm II_1$-factor for which there is an isometric embedding $\bb S^\infty\hookrightarrow U(M)$. We currently do not know whether $\bb S^\infty$ embeds isometrically in the unitary group of \emph{any} $\rm II_1$-factor. The existence of such an embedding ought to have striking consequences as the following proposition, which is similar in  spirit, demonstrates.

\begin{prop}
Suppose that $M$ is a separable II$_1$  factor belonging to the class $\kop$.  Further suppose that, for each $n$, the $n$-dimensional complex spheres $\bb S^{n}$ isometrically embed in $U(M)$ with respect to the natural $\bb Z_4$-actions.  Then $M$ is a locally universal $\rm II_1$-factor, that is, every separable II$_1$ factor embeds into an ultrapower of $M$. In particular, if, for each $n$, the $n$-dimensional complex spheres $\bb S^{n}$ isometrically embed in $U(\R)$ with respect to the natural $\bb Z_4$-actions, then Connes' embedding problem has a positive answer.
 \end{prop}

 \begin{proof}
Suppose that $M$ satisfies the assumption of the proposition and let $N$ be an II$_1$ factor.  Let $F$ be any finite subset of $U(N)$. Then choosing an orthogonal projection $P$ onto a suitably large finite-dimensional subspace so that $\|P(u)\|> 1- \e$ for all $u\in F \cup iF$, we can correct to an (effective in) $\e$-almost $\bb Z_4$-embedding of $F$ into some $\bb S^n$, and therefore also in $U(M)$. But $\bb Z_4$-embeddings preserve inner products, whence pairs of inner products in $F$ can be modeled arbitrarily well in $U(M)$.  As above, Kirchberg's argument shows that $N$ embeds in $M^\uu$.
 \end{proof}


We now remark how our main result recasts Kirchberg's characterization of $\R^\om$-embeddability in a game-theoretical light. Let $(A,\tr)$ be an arbitrary tracial C$^*$-algebra which we view as a normed space with respect to the $2$-norm. To introduce a bit of terminology, we say that a subspace $E\subset A$ is \emph{$\e$-almost representable} in $\R$ if there exists a subspace $F\subset \R$ and a linear bijection $T: E\to F$ so that $\nor{T}, \nor{T^{-1}}\leq 1+\e$ and $T(E\cap (A)_1)\subset_\e F\cap (\R)_1$. Then by Proposition 4.6 in \cite{kirch}, $A$ is $\R^\om$-embeddable if and only if for every $\e>0$, every finite-dimensional subspace of $A$ is $\e$-representable in $\R$.

Let us introduce the following ``one-sided, one-round game'' $\fr G_{\R}(n,\e)$ for which the winning condition is that, for all $u_1,\dotsc,u_n\in U(A)$ which are linearly independent, there exist $n$ unitaries $v_1,\dotsc,v_n\in U(\R)$ so that the map $$T: \operatorname{span}\{u_1,\dotsc,u_n\}\to \operatorname{span}\{v_1,\dotsc,v_n\}$$ defined by $T(u_i) = v_i$ satisfies $\nor{T},\nor{T^{-1}}\leq 1+\e$.

\begin{prop} There is a constant $N = N(n,\e)$ so that every $n$-dimensional subspace $E$ of any tracial $C^*$-algebra $(A,\tr)$ is $\e$-almost representable in $\R$ if $\fr G_{\R}(N,\e/4)$ is winnable.
\end{prop}

\begin{proof} We first claim that there is a uniform constant $K(n,\e)$ so that for every $n$-dimensional subspace $E\subset A$ of any tracial $C^*$-algebra $(A,\tr)$ there exists a set of unitaries ${\bar u} =\{u_1,\dotsc,u_l\}\subset U(A)$ with $l\leq K$ so that every element of $E\cap (A)_1$ is $\e$-approximated in $2$-norm by a convex combination of elements of $\bar u$.

Indeed, choose an $\e/2$-net $x_1, \dotsc, x_m\in E\cap (A)_1$. The cardinality of such a net is bounded in particular by the $\e/4$-covering number of the unit ball in $\ell^2_n$. We may perturb each $x_i$ so that $\nor{x_i}<1 - \e/4$ and still have an $\e$-net for $E\cap (A)_1$. By the main result of \cite{popa-rd} there is a constant $C$ depending only on $\e$ so that each $x_i$ is a convex combination of at most $C$ unitaries in $U(A)$, whence the claim follows.

We next claim that if $A$ is infinite-dimensional and if $E\subset A$ is a finite-dimensional subspace, then for every $\e>0$ and $u\in U(A)$ there exists $u'\in U(A)$ with $\nor{u - u'}_2<\e$ and so that $u'$ is linearly independent from $E$.  To see this, let $P_E: L^2(A)\to E$ be the orthogonal projection onto $E$. By the Kaplansky density theorem, we have that $U(A)$ is $2$-norm dense in $U(A'')$. Since $M := A''\subset \cal B(L^2(A,\tr))$ is infinite-dimensional, it contains a diffuse abelian subalgebra. Therefore, there is a projection $p\in M$ with trace $\tr(p) = 1-\e^2/2$ and a sequence of unitaries $v_n\in U(M)$ so that $v_n\to p$ weakly. Since $P_E$ is a finite-rank operator, we thus have that $P_E(uv_n)\to P_E(up)$ strongly, whence $\nor{P_E(uv_n)}_2 \to \nor{P_E(up)}_2\leq \nor{p}_2 = \sqrt{1 - \e^2/2}$. It is now easy to see that choosing $n$ sufficiently large and $u'\in U(A)$ sufficiently close to $uv_n$ works.

We now can proceed with the proof of the proposition. Let $E = \operatorname{span}\{u_1,\dotsc,u_n\}\subset A$. (Every $n$-dimensional subspace of a $C^*$-algebra is a subspace of a space spanned by at most $4n$ unitaries, so we may assume this is the case without loss of generality.) By the previous claims, we can extend $u_1, \dotsc, u_n$ to $u_1,\ldots,u_n, u_{n+1}, \dotsc, u_s$ ($s\leq n + K(n,\e)$) to a complete collection of linearly independent unitaries so that all elements in $E\cap (A)_1$ are $2\e$-approximated in $2$-norm by a convex combination of unitaries in the collection. If $\fr G_R(s,\e/4)$ is winnable, then it is easy to check that for $S = T|_E$ we have that $S(E\cap (A)_1) \subset_{\e} S(E)\cap (\R)_1$, and we are done.
\end{proof}

\begin{prob} Let $\cal C\subset \ell^2_n$ be a convex subset of the unit ball in $n$-dimensional Hilbert space. For every $\e>0$ does there exists a $\rm II_1$-factor $M$ so that $(\ell^2_n,\cal C)$ is $\e$-represented in $M$? Can one always choose a locally universal $\rm II_1$-factor (in the sense of \cite{mtoa3}) or even $\R$?
\end{prob}

\bibliographystyle{amsplain}

\begin{thebibliography}{10}















\bibitem{ben} I. Ben Yaacov, {\it Continuous first order logic for unbounded metric structures}, J. Math. Log. {\bf 8} (2008), 197-223.

\bibitem{bbhu} I. Ben Yaacov, A. Berenstein, C.W. Henson, and A. Usvyatsov, \textit{Model theory for metric structures}, Model theory with applications to algebra and analysis. vol. 2, pp. 315-427, London Math. Soc. Lecture Note Ser. {\bf 350}, Cambridge UP, Cambridge, 2008.

\bibitem{Connes} A. Connes, \textit{Sur la classification des factuers de type II}, C.R. Acad. Sci. Paris S\'er. A-B \textbf{281} (1975), A13-A15.


\bibitem{DL} J. Dixmier and E.C. Lance, \textit{Deux nouveaux facteurs de type II$_1$}, Invent. Math. \textbf{7} (1969), 226-234.


\bibitem{dye} H.A. Dye, \textit{On the geometry of projections in certain operator algebras}, Ann. of Math. (2) \textbf{61} (1955), 73-89.

\bibitem{ealygold} C. Ealy and I. Goldbring, \textit{Thorn-forking in continuous logic}, Journal of Symbolic Logic \textbf{77} (2012), 63-93.


\bibitem{mtoa1} I. Farah, B. Hart, and D. Sherman, \emph{Model theory of operator algebras I:  stability}, Bulletin of the London Math Society \textbf{45} (2013), 825-838.
\bibitem{mtoa2} I. Farah, B. Hart, and D. Sherman, \emph{Model theory of operator algebras II:  model theory}, Israel J. Math. (to appear). arXiv:1004.0741.

\bibitem{mtoa3} I. Farah, B. Hart, and D. Sherman, \emph{Model theory of operator algebras III:  elementary equivalence and II$_1$ factors}, Bull. London Math. Soc. (to appear). arXiv:1111.0998.


\bibitem{jordan} H. Hanche-Olsen and E. St{\o}rmer, {\it Jordan operator algebras}, Monographs and Studies in Mathematics {\bf 21}, Pitman, Boston, MA, 1984.

\bibitem{bradd} B. Hart, {\it Continuous model theory and its applications}, online lecture notes.  Available at \texttt{http://ms.mcmaster.ca/~bradd/courses/math712/index.html}.


\bibitem{heinrich-henson} S. Heinrich and C.W. Henson, {\it Banach space model theory II. Isomorphic equivalence}, Math. Nachr. {\bf 125} (1986), 301--317.

\bibitem{henson-iovino} C.W. Henson and J. Iovino, {\it Ultraproducts in analysis}, in Analysis and Logic (Mons, 1977), pp. 1--110, London Math. Soc. Lecture Note Ser. {\bf 262}, Cambridge UP, Cambridge, 2002.

\bibitem{henson-moore} C.W. Henson and L.C. Moore, {\it Nonstandard analysis and the theory of Banach spaces}, Nonstandard analysis-recent developments (Victoria, B.C., 1980), 27-112, Lecture Notes in Math., 983, Springer, Berlin, 1983.

\bibitem{kirch} E. Kirchberg, {\it On non-semisplit extensions, tensor products and exactness of group C$^*$-algebras}, Invent. Math. \textbf{112} (1993), 449-489.


\bibitem{nvt} L. Nguyen Van Th\'e, {\it Structural Ramsey theory of metric spaces and topological dynamics of isometry groups}, Mem. Amer. Math. Soc. {\bf 968} (2010).

\bibitem{ozawa} N. Ozawa, {\it There is no separable universal $\rm II_1$-factor}, Proc. Amer. Math. Soc. {\bf 132} (2004), 487--490.

\bibitem{popa-rd} S. Popa, {\it On the Russo--Dye theorem}, Michigan Math. J. {\bf 28} (1981) 311--315.

\bibitem{tak} M. Takesaki, {\it Theory of Operator Algebras III}, Encyclopaedia of of Mathematical Sciences {\bf 127}, Springer, Berlin, 2003, xix + 548 pp.
































\end{thebibliography}

\end{document}